\documentclass[12pt]{article}
\usepackage{sty-jir}
\usepackage[margin=1in]{geometry}

\newcommand{\Variety}{\mathcal{V}}

\newcommand{\weight}{w}
\newcommand{\profilePath}{\beta}
\newcommand{\profileMap}{\Phi_{\weight}}
\newcommand{\OpenSimplex}{\Delta_{k-1}^{\circ}}

\newcommand{\QQ}{\mathbb{Q}}
\newcommand{\RR}{\mathbb{R}}
\newcommand{\CC}{\mathbb{C}}
\newcommand{\ZZ}{\mathbb{Z}}

\DeclareMathOperator{\Crit}{Crit}
\DeclareMathOperator{\trdeg}{trdeg}
\DeclareMathOperator{\rank}{rank}
\DeclareMathOperator{\Frac}{Frac}

\newcommand{\RidgeCorr}{{\mathfrak{L}}}
\newcommand{\closedalgebraicset}{X}

\newcommand{\cO}{\mathcal O}

\newcommand{\OI}{\cO(I)}

\newcommand{\MI}{\Frac{\cO(I)}}

\newcommand{\MIC}{\MI(\sqrt{-1})}

\newcommand{\positiveprofiles}{\PP_{>0}^{k-1}}
\newcommand{\interval}{I}
\newcommand{\xPath}{\eta}

\newcommand{\domainDensity}{\mathbb{R}^d}

\newcommand{\varx}{{x}}
\newcommand{\varw}{\beta}

\newcommand{\indet}{y}

\newcommand{\mixturedensity}{f^{(\weight)}}

\newcommand{\selectthefs}{f_1,\ldots,f_k}

\newcommand{\PP}{\mathbb{P}}

\newcommand{\xprojection}[2]{\pi_x:{#1}\longrightarrow{#2}}
\newcommand{\betaprojection}[2]{\pi_\profilePath:{#1}\longrightarrow{#2}}
\newcommand{\wprojection}[2]{\pi_w:{#1}\longrightarrow{#2}}

\newcommand{\graphprofilemap}{\operatorname{Graph}(\profileMap)}

\newcommand{\profileopenset}{U}

\title{
A ridgeline correspondence criterion:
\\
the number of modes of a Gaussian mixture is finite
}

\author{Carlos Améndola and Jose Israel Rodriguez}
\date{}
\begin{document}
\maketitle 

\begin{abstract}

 We prove that every finite multivariate Gaussian mixture density has only finitely many modes. 
 Our approach combines an algebraic formulation of the ridgeline theory of Ray and Lindsay (2005) with a transcendence-degree argument based on Ax's functional-transcendence theorem to bound the cardinality of the set of critical points. Our techniques extend recent work by Wang (2026), who used Ax's theorem together with real-analytic curve selection to prove finiteness of the critical set of homoscedastic Gaussian mixtures. We introduce the ridgeline correspondence and use it to obtain a finiteness result that applies to arbitrary heteroscedastic Gaussian mixtures. Our framework also establishes finiteness of the number of modes for additional classes of polynomial-exponential mixtures and generalized Gaussian mixtures. 

\end{abstract}

\section{Introduction}

Gaussian mixtures have been widely studied for their versatility and universal approximation properties, see e.g. \cite[Ch. 3]{Goodfellow-et-al-2016}.
They 
also model the presence of several normally distributed populations and are useful in clustering \cite[Sec. 9.2]{bishop2006pattern}. Notably, the \emph{mean-shift algorithm} with Gaussian kernel finds modes, i.e., local maxima, of a Gaussian mixture \cite{mean-shift}. They also feature in nonparametric maximum likelihood estimation~\cite{NPMLE}.

While a one-dimensional mixture of $k$ Gaussians can exhibit at most $k$ modes, the situation becomes more complex for mixtures in $\RR^d$ \cite{carreira2003number}.

Let $m(d,k)$ denote
the maximum number of modes among all $d$-dimensional Gaussian mixtures with $k$ components. 
Only a
few values of $m(d,k)$ are known. 
Apart from the previously mentioned $m(1,k)=k$ and the obvious $m(d,1)=1$,
Ray and Ren showed
that $m(d,2)=d+1$ 
\cite{Ray-Ren}. 
In fact, already $m(2,3)$ is unknown, with the best lower bound improved recently to $m(2,3) \geq 7$ by Kabata, Matsumoto and Okuno \cite{kabata2026least}. 

One natural 
question is whether $m(d,k) < \infty$ for all $d,k$. In this paper we answer this question affirmatively, as the main application of the theory developed here. The main inspiration comes from the \emph{ridgeline theory} introduced by Ray and Lindsay \cite{Ray-Lindsay} while studying critical points of Gaussian mixtures, and from the recent connection to transcendence theory first made by Wang \cite{wang2026finitegaussianmixturesfiniteness},
who proved
the finiteness of critical points in the special case of homoscedastic Gaussian mixtures.

Our main result is a ridgeline criterion for concluding finitely many critical points, \Cref{thm:criterion}, and as a main application we obtain the following.

\begin{theorem}
[The number of modes of a Gaussian mixture is finite]
\label{theorem:quadratic}
Let
\[
f(x)=\sum_{i=1}^{k} w_i \frac{1}{\sqrt{\det(2\pi\Sigma_i)}} \exp\!\left( -\frac{1}{2} (x-\mu_i)^\top \Sigma_i^{-1} (x-\mu_i) \right),
\]
where $\mu_i \in \RR^d$, $\Sigma_i$ is symmetric positive definite and $w_i>0$ with $w_1+\dots+w_k=1$.
Then $f$ has finitely many critical points.
In particular, every Gaussian mixture has finitely many modes. 
\end{theorem}

As an outcome, the conditional upper bounds of Améndola-Engstrom-Haase \cite{Amendola-Engstrom-Haase} and the recently improved ones by Nguyen \cite{nguyen2026bounds} now apply unconditionally in all cases. 
In this recent flurry of activity on the problem it is also worth noting the unconditional result of Okuno and Kabata
\cite{okuno2026mixtures} proving that there are at most eight modes for homoscedastic mixtures of three Gaussians. 

\medskip

 \medskip
This paper is organized as follows.
In \Cref{sec:ridgeline}, 
    we introduce the ridgeline correspondence for polynomial--exponential mixtures and show that the critical points of a mixture are encoded by its intersection with the graph of a real-analytic profile map. 
In \Cref{sec:analytic-transcendence-ingredients}, 
    we develop the analytic and transcendence degree  ingredients needed for the main argument.
In \Cref{sec:ax-ridgeline-bridge}, 
    we combine these ingredients to prove the 
    ridgeline finiteness criterion, 
    \Cref{thm:criterion}. 
In \Cref{sec:gaussian-mixtures}, 
    we apply this criterion to arbitrary Gaussian mixtures, proving \Cref{theorem:quadratic}. 
To conclude, 
    we discuss applications to nonparametric maximum likelihood estimation (NPMLE),    
    extensions
    to other polynomial--exponential mixture families, 
    examples demonstrating the necessity of the hypotheses, and open questions concerning the  number $m(d,k)$ of Gaussian mixture modes.

    \section{Ridgeline correspondence and critical point incidence}\label{sec:ridgeline}
In this section we define the main algebraic object of the \demph{ridgeline correspondence} associated to a family of mixture densities. 
This is inspired by the 
\demph{likelihood correspondence} studied in algebraic statistics~
\cite{MR3329087,kahle2026likelihood,MRWW2024-HS-Conjecture}. 
In the Gaussian mixture case, the ridgeline correspondence 
is directly related to the
\demph{ridgeline manifold} introduced by Ray and Lindsay~\cite{Ray-Lindsay}.

\medskip

\subsection{The family of mixture densities}

Fix strictly positive continuously differentiable probability
densities
$f_1,\ldots,f_k:
\RR^d
\longrightarrow
\RR_{>0}.
$
For \demph{weights}
$\weight
=
(\weight_1,\ldots,\weight_k)$
in the $(k-1)$-dimensional \demph{open probability simplex},
\[
\OpenSimplex:= \{ (w_1,\dots,w_k)\in \RR^k:  w_i>0,\; w_1+\cdots+w_k=1
\},
\]
define the corresponding \demph{mixture density} by
\begin{equation}
\label{eq:mixture}
\mixturedensity(x)
=
\sum_{i=1}^{k}
\weight_i f_i(x),
\qquad
x\in\RR^d.
\end{equation}
The component densities
$f_1,\ldots,f_k$
are fixed, while the weight vector $\weight$ selects one member of
the family of densities
\[
\left\{
\mixturedensity:
\weight\in\OpenSimplex
\right\}.
\]
The \demph{score function} 
$s_i:\domainDensity\to \domainDensity$
of the component density $f_i:\domainDensity\to \RR_{>0}$ is
\begin{equation}\label{eq:score-field}
s_i(x)
:=
\nabla\log f_i(x)
=
\frac{\nabla f_i(x)}{f_i(x)}.
\end{equation}
Since $\nabla f_i(x)=f_i(x)s_i(x)$,
the \demph{critical set} of the mixture \eqref{eq:mixture} is
\begin{align}
\label{eq:crit-f}
\Crit(\mixturedensity)
&\notag
:=
\left\{
x\in\RR^d:
\nabla\mixturedensity(x)=0
\right\}\\
&=
\left\{
x\in\RR^d:
\sum_{i=1}^{k}
\weight_i f_i(x)s_i(x)=0
\right\}.
\end{align}

\subsection{The algebraic ridgeline correspondence}

    We now specialize to polynomial--exponential component densities.
Suppose that
\begin{equation}\label{eq:nice-f-i-h-i-polynomial}
f_i(x)=e^{h_i(x)},
\qquad
h_i\in\RR[x_1,\ldots,x_d],
\qquad
i=1,\ldots,k.
\end{equation}
Here constant terms are allowed in the polynomials $h_i$ . In
particular, this notation includes the normalization constants that
occur in Gaussian densities.

The score functions are the polynomial maps
$s_i(x)
=
\nabla\log f_i(x)
=
\nabla h_i(x).
$
Thus,
\begin{align}
\label{eq:polynomial-exponential-gradient}
\nabla\mixturedensity(x)
&=
\sum_{i=1}^k
\weight_i e^{h_i(x)}\nabla h_i(x).
\end{align}
Although the score vectors are polynomial in $x$ , their coefficients
in~\eqref{eq:polynomial-exponential-gradient} involve exponentials.

Replacing $\weight_i e^{h_i(x)}$ with the indeterminate $\profilePath_i$ gives  
the 
\demph{ridgeline correspondence polynomials} 
\begin{align}\label{eq:define-G-polynomials}
\begin{bmatrix}
G_1(x,\profilePath)\\
\vdots\\
G_d(x,\profilePath)
\end{bmatrix}
:=
\sum_{i=1}^k
\profilePath_i
\nabla
h_i(x)
\end{align}
in the polynomial ring
$S
:=
\CC[
x_1,\ldots,x_d,
\profilePath_1,\ldots,\profilePath_k
].$
Each $G_j$ is homogeneous in the profile coordinates $\profilePath$, so if $\profilePath_1+\cdots+\profilePath_k\neq 0$, 
then we assume  $\profilePath_1+\cdots+\profilePath_k = 1$.

\begin{definition}
\label{def:ridgeline-ideal-correspondence}
Let $f_1,\dots,f_k$ be as in \eqref{eq:nice-f-i-h-i-polynomial}
and 
let $G_1,\ldots,G_d$ be the polynomials defined in \eqref{eq:define-G-polynomials}.
The \demph{ridgeline correspondence ideal} associated with 
$\selectthefs$ is the saturated ideal
\begin{equation}
\label{eq:ridgeline-ideal}
I_{\mathrm{ridge}}
:=
\langle G_1,\dots,G_d\rangle 
:
(\profilePath_1\cdots\profilePath_k)^\infty
\subseteq S.
\end{equation}
The \demph{ridgeline correspondence} of
$\selectthefs$ is the
vanishing set of $I_{\mathrm{ridge}}$, and we write
\begin{equation}
\label{eq:ridgeline-correspondence}
\RidgeCorr
\bigl(\selectthefs\bigr)
:=
\Variety_{\CC^d\times\PP_{\CC}^{k-1}}
\left(
I_{\mathrm{ridge}}
\right)
\subseteq
\CC^d\times\PP_{\CC}^{k-1}.
\end{equation}
\end{definition}

\begin{remark}[Saturation on the nonzero profile locus] 
\label{remark:saturation-nonzero-profile} Suppose that 
\[ [\profilePath]\in\PP_{\CC}^{k-1} \qquad\text{and}\qquad \profilePath_1\cdots\profilePath_k\neq0. \] Then 
$(x,[\profilePath]) \in \RidgeCorr(f_1,\ldots,f_k)$
if and only if 
$G_1(x,\profilePath) = \cdots = G_d(x,\profilePath) = 0.$ 
Thus saturation does not alter the solutions on the nonzero profile locus $\left\{ [\profilePath]\in\PP_{\CC}^{k-1}: \profilePath_1\cdots\profilePath_k\neq0 \right\}.$
\end{remark}

\subsection{The profile map and critical point incidence}

\begin{definition}[Positive profiles and the profile map]
The \demph{positive real profile locus} is
\[
\positiveprofiles
:=
\left\{
[\profilePath]\in\PP_{\RR}^{k-1}:
[\profilePath]
\text{ has a representative with all coordinates positive}
\right\}.
\]
For
$ \weight\in\OpenSimplex$ , define the \demph{profile map with respect
to $\weight$ } by
\begin{equation}
\label{eq:define-profile-map-w}
\profileMap:
\RR^d\longrightarrow\positiveprofiles,
\qquad
x\longmapsto
[
\weight_1f_1(x):
\cdots:
\weight_kf_k(x)
].
\end{equation}
\end{definition}

The \demph{profile graph} is
\begin{equation}
\label{eq:graph-of-profile-map}
\graphprofilemap
:=
\left\{
(x,\profileMap(x)):
x\in\RR^d
\right\}
\subseteq
\RR^d\times\positiveprofiles.
\end{equation}
We also regard the profile graph as a subset of the complex ambient
space through the natural inclusions
\[
\graphprofilemap
\subseteq
\RR^d\times\positiveprofiles
\subseteq
\RR^d\times\PP_{\RR}^{k-1}
\subseteq
\CC^d\times\PP_{\CC}^{k-1}.
\]

\begin{example}
In the polynomial--exponential setting, the profile map is real analytic, and hence its graph is a real-analytic submanifold of $\RR^d\times\positiveprofiles$.
Indeed, if $f_i=e^{h_i}$ , then
\begin{equation}
\label{eq:profile-graph-polynomial-exponential}
\graphprofilemap
=
\left\{
\left(
x,
[
\weight_1e^{h_1(x)}:
\cdots:
\weight_ke^{h_k(x)}
]
\right):
x\in\RR^d
\right\}.
\vspace{-12pt}
\end{equation}
\end{example}

\begin{notation}
Let
\[
\betaprojection
    {\CC^d\times\PP_{\CC}^{k-1}}
    {\PP_{\CC}^{k-1}}
,
\qquad
(x,[\profilePath])
\longmapsto
[\profilePath]
\]
and
\[
\xprojection
    {\CC^d\times\PP_{\CC}^{k-1}}
    {\CC^d}
,
\qquad
(x,[\profilePath])
\longmapsto
x
\]
When restricting either projection, we indicate its domain and
codomain.
\end{notation}

The set of critical points of $\mixturedensity$
admit the following description.

    \begin{proposition}[Critical point incidence] 
\label{prop:critical-point-incidence} 
Suppose that $f_i=e^{h_i}$ with $h_i\in\RR[x_1,\ldots,x_d]$ for $i=1,\ldots,k$ .
For every $\weight\in\OpenSimplex$, the  projection
\[ 
\xprojection
    {\RidgeCorr(f_1,\ldots,f_k) \cap \graphprofilemap}
    {\Crit(\mixturedensity)} 
\] 
 induces a bijection.
Its inverse is 
$    x \longmapsto (x,\profileMap(x)).$ 
Equivalently, 
\[ 
    x\in\Crit(\mixturedensity) \quad\
    \text{ if and only if }
    \quad (x,\profileMap(x)) \in \RidgeCorr(\selectthefs). 
    \] \end{proposition}

\begin{proof}

Fix $x\in\RR^d$ and use the positive homogeneous representative \[ (\profilePath_1,\ldots,\profilePath_k) = \bigl( \weight_1f_1(x),\ldots,\weight_kf_k(x) \bigr) \] of $\profileMap(x)$ . 
Substituting, 
we get, 
\begin{align*} G_j(x,\profileMap(x)) 
    &:= 
    G_j\bigl( x, (\weight_1f_1(x),\ldots,\weight_kf_k(x)) \bigr)\\ &= \frac{\partial\mixturedensity}{\partial x_j}(x). \end{align*}
Therefore,
\[
\nabla\mixturedensity(x)=0
\quad\Longleftrightarrow\quad
G_j(x,\profileMap(x))=0
\quad\text{for every }j=1,\ldots,d.
\]
Every profile coordinate
$ \weight_i e^{h_i(x)}$ is strictly positive, so
Remark~\ref{remark:saturation-nonzero-profile} shows that saturation
does not alter this condition. Hence
\[
x\in\Crit(\mixturedensity)
\quad\Longleftrightarrow\quad
(x,\profileMap(x))
\in
\RidgeCorr(f_1,\ldots,f_k).
\]
Finally, because $\graphprofilemap$ is the graph of a function, the
profile coordinate is uniquely determined by $x$ . The indicated
projection is therefore a bijection with the stated inverse.
\end{proof}

\subsection{Ridgeline loci in the $x$ -space}
 
\begin{definition}[Ridgeline loci]
The \emph{complex ridgeline locus}
$\mathcal R(f_1,\ldots,f_k)$
is
the image of 
\[\xprojection
    {\RidgeCorr(\selectthefs)}
    {\CC^d}.
\]
The image of  
\[\xprojection
    {
    \RidgeCorr(\selectthefs)
    \cap
    \left(
    \RR^d\times\positiveprofiles
    \right)
    }
    {\RR^d}
\]
is the \demph{positive ridgeline locus}
$\mathcal R_{>0}(\selectthefs)$.

\end{definition}

Because 
$    \RidgeCorr(\selectthefs)
\subset \CC^d\times\PP_{\CC}^{k-1}$
and
$ \PP_{\CC}^{k-1}$ is projective, 
$ \mathcal R(f_1,\ldots,f_k)$ is a closed algebraic subset of
$ \CC^d$ .
We call $\mathcal R(f_1,\ldots,f_k)$ the \demph{ridgeline variety}.

We get the following containment.
\begin{corollary}[Ridgeline containment]
\label{cor:ridgeline-containment}
For every $\weight\in\OpenSimplex$ ,
\[
\operatorname{Modes}(\mixturedensity)
\subseteq
\Crit(\mixturedensity)
\subseteq
\mathcal R_{>0}(f_1,\ldots,f_k).
\]
\end{corollary}

\begin{proof}
Every local maximum of the differentiable function
$ \mixturedensity$ is a critical point. By
Proposition~\ref{prop:critical-point-incidence}, every critical point
has a positive projective profile and therefore lies in the positive
ridgeline locus.
\end{proof}

\begin{remark}
The terminology of \emph{ridgeline} in our definitions is inspired by the work of Ray and Lindsay~\cite{Ray-Lindsay}, who  
introduced the \emph{ridgeline manifold} to study the critical points of Gaussian mixtures. In the Gaussian case, the closure of our positive ridgeline locus coincides with their ridgeline manifold, while the ridgeline correspondence used here retains the projective profile coordinates rather than immediately projecting to the sample space. 
\end{remark}

\subsection{Example: two univariate Gaussian components}

Consider the univariate Gaussian mixture
\begin{equation}
\label{eq:two-univariate-gaussian-mixture}
f(x)
=
w_1
\frac{1}{\sqrt{2\pi\sigma_1^2}}
\exp\left(
-\frac{(x-\mu_1)^2}{2\sigma_1^2}
\right)
+
w_2
\frac{1}{\sqrt{2\pi\sigma_2^2}}
\exp\left(
-\frac{(x-\mu_2)^2}{2\sigma_2^2}
\right),
\end{equation}
where
\[
w_1,w_2>0,
\qquad
w_1+w_2=1.
\]
The $i$ -th component has mean $\mu_i$ and variance
$ \sigma_i^2$ . Write
\[
h_i(x)
=
-\frac{(x-\mu_i)^2}{2\sigma_i^2}
-\frac12\log(2\pi\sigma_i^2),
\qquad
i=1,2.
\]
Then $f_i=e^{h_i}$ , and
\[
h_i'(x)
=
-\frac{x-\mu_i}{\sigma_i^2}.
\]

The ridgeline correspondence is
\begin{align}
\RidgeCorr(f_1,f_2)
=
\biggl\{
(x,[\profilePath_1:\profilePath_2])
\in
\CC\times\PP_{\CC}^1:
\profilePath_1h_1'(x)
+
\profilePath_2h_2'(x)
=
0
\biggr\}.
\label{eq:two-gaussian-ridgeline-correspondence}
\end{align}
Its defining equation may be written as
\[
\profilePath_1
\frac{x-\mu_1}{\sigma_1^2}
+
\profilePath_2
\frac{x-\mu_2}{\sigma_2^2}
=
0.
\]
Whenever
\[
\frac{\profilePath_1}{\sigma_1^2}
+
\frac{\profilePath_2}{\sigma_2^2}
\neq 0,
\]
we can solve for $x$ :
\begin{equation}
x
=
\frac{
\profilePath_1\sigma_2^2\mu_1
+
\profilePath_2\sigma_1^2\mu_2
}{
\profilePath_1\sigma_2^2
+
\profilePath_2\sigma_1^2
}
=
\frac{
\profilePath_1\sigma_2^2
}{
\profilePath_1\sigma_2^2
+
\profilePath_2\sigma_1^2
}
\mu_1
+
\frac{
\profilePath_2\sigma_1^2
}{
\profilePath_1\sigma_2^2
+
\profilePath_2\sigma_1^2
}
\mu_2.
\label{eq:two-gaussian-ridgeline-parameterization}
\end{equation}
If
$ \profilePath_1,\profilePath_2>0$ , the two coefficients on the
right-hand side are positive and sum to one. Therefore, if
$ \mu_1\neq\mu_2$ , every point in the positive ridgeline locus lies
strictly between the two component means. In this case,
\[
\mathcal R_{>0}(f_1,f_2)
=
\left(
\min\{\mu_1,\mu_2\},
\max\{\mu_1,\mu_2\}
\right).
\]
For fixed weights $(w_1,w_2)$ , the profile map is
\begin{equation}
\label{eq:two-gaussian-profile-map}
\profileMap(x)
=
[
w_1f_1(x):
w_2f_2(x)
].
\end{equation}
The critical points of the mixture are precisely the
values of $x$ for which
\[
\left(
x,
[
w_1f_1(x):
w_2f_2(x)
]
\right)
\in
\RidgeCorr(f_1,f_2).
\]

Figure~\ref{fig:two-gaussian-ridgeline-profile} illustrates this
incidence construction.
The blue curve is the plot of the mixture density. In the profile-coordinate plot, the orange
curve represents the ridgeline correspondence and the green curve
represents the profile graph. 
The $x$ -coordinates of the
intersection points of the orange and green curves are precisely the
critical points of the blue mixture density. The graph of the mixture
then distinguishes which of these critical points are local maxima and
which are local minima.

\begin{figure}[htb!]
    \centering
\includegraphics[width=0.5\linewidth]{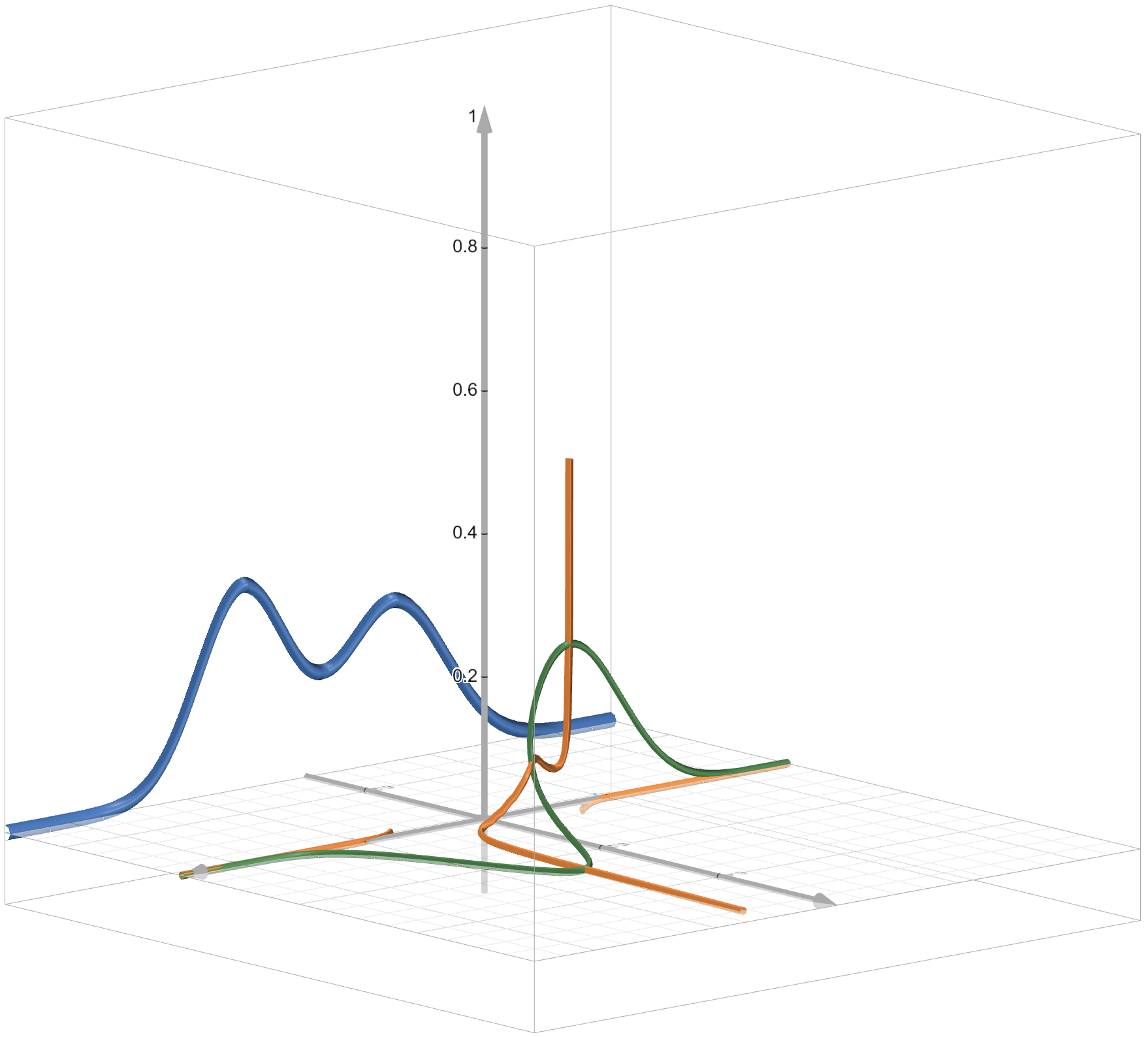}    
    \caption{The ridgeline correspondence and profile graph for a
    mixture of two univariate Gaussian densities. Intersections of the
    ridgeline correspondence and the profile graph encode the critical
    points of the mixture.}
    \label{fig:two-gaussian-ridgeline-profile}
\end{figure}

\section{Analytic and transcendence degree background} \label{sec:analytic-transcendence-ingredients} 
This section recalls the analytic and transcendence degree results needed for the proof of the ridgeline finiteness criterion. In particular, we recall the power-series form of Ax's theorem and derive the one-variable analytic consequence used in the sequel.

\subsection{Analytic paths in the incidence locus} 

We now describe the profile coordinates along a real-analytic path in the incidence locus. 
We then introduce the field of real-analytic
functions on an interval and derive the one-variable consequence of Ax's theorem needed in the sequel.

        Suppose that
\[
\gamma:
\interval
\longrightarrow
\RidgeCorr(f_1,\ldots,f_k)
\cap
\graphprofilemap
\subseteq
\RR^d\times\positiveprofiles
\]
is a real-analytic path. Write
\[
\gamma(t)
=
\bigl(x(t),\profilePath(t)\bigr),
\]
where
$x:
\interval
\longrightarrow
\RR^d$
and
$
\profilePath:
\interval
\longrightarrow
\OpenSimplex
$

Since $\profilePath(t)=\profileMap(x(t))$ , 
we have
\begin{equation}
\label{eq:normalized-profile-along-path}
\profilePath_i(t)
=
\frac{\weight_i f_i(x(t))}
{\mixturedensity(x(t))},
\qquad
i=1,\ldots,k.
\end{equation}

\begin{example}
In the polynomial exponential setting  
$f_i=e^{h_i}$, and $h_i\in \RR[x_1,\dots,x_d]$.
Set $v_i:=h_i\circ x\in \OI$.
Then
\[
\profilePath_i
=
\frac{\weight_i e^{v_i}}
{\sum_{\ell=1}^k\weight_\ell e^{v_\ell}},
\qquad i=1,\dots,k
\]
In particular, since $\beta_k(t)>0$ for all $t$,
we have
\begin{equation}
\label{eq:profile-ratio-exponential-relation}
\frac{\profilePath_i}{\profilePath_k}
=
\frac{\weight_i}{\weight_k}
 e^{v_i-v_k}
=
\frac{\weight_i}{\weight_k}
\exp\bigl((h_i-h_k)\circ x\bigr),
\qquad
i=1,\ldots,k-1.
\end{equation}
These identities connect the algebraic profile coordinates of the
ridgeline correspondence with exponentials of analytic functions.
\end{example}

\subsection{Fields of real-analytic functions} 

Let $\interval\subseteq\RR$ be a nonempty open interval,
and fix a base point $t_0\in\interval$.
Let $\OI$ be the ring of
real-valued real-analytic functions on $\interval$.
This is an
integral domain, and  we write
$\MI$
for its fraction field.

Let $z$ be a formal variable. Taylor expansion at $t_0$ defines an
injective $\RR$-algebra homomorphism
\begin{equation}
\label{eq:taylor-embedding}
T_{t_0}:
\OI
\longrightarrow
\RR[[z]],
\qquad
f
\longmapsto
\sum_{m=0}^{\infty}
\frac{f^{(m)}(t_0)}{m!}z^m.
\end{equation}
Indeed, if $T_{t_0}(f)=0$, then $f$ vanishes on a neighborhood of
$t_0$ and hence, by the identity theorem for analytic functions, on all of $\interval$.

    Since $T_{t_0}$ is injective and $\RR[[z]]$ is an integral domain,
it extends uniquely to an embedding of fraction fields
\begin{equation}
\label{eq:analytic-field-laurent-embedding}
\MI
\hookrightarrow
\Frac\bigl(\RR[[z]]\bigr)
=:
\RR((z)).
\end{equation}
After adjoining $\sqrt{-1}$ , the Taylor embedding extends
to
\begin{equation}
\label{eq:complexified-analytic-field-embedding}
\MIC
=
\MI(\sqrt{-1})
\hookrightarrow
\CC((z)).
\end{equation}
We use this embedding to regard the complex fields generated by the
analytic functions below as subfields of $\CC((z))$ .

By properties of Taylor expansion, 
for every $u\in\OI$,
\begin{equation}
\label{eq:taylor-differentiation}
T_{t_0}(u')
=
\frac{d}{dz}T_{t_0}(u)
\end{equation}
and
\begin{equation}
\label{eq:taylor-exponentiation}
T_{t_0}(e^u)=
\exp\bigl(T_{t_0}(u)\bigr).
\end{equation}
For an analytic path
$x:
\interval
\longrightarrow
\RR^d,$
the polynomial functions restrict to analytic functions
\begin{equation}
v_i(t)
:=
h_i(x(t))
\in
\OI.
\end{equation}

We now recall the definition of transcendence degree.
\begin{definition}[Algebraic independence and transcendence degree]
\label{def:transcendence-degree}
Let $E\subseteq F$ be a field extension. Elements
$u_1,\ldots,u_m\in F$
are \demph{algebraically independent over $E$} if there is no nonzero
polynomial
$P\in E[\indet_1,\ldots,\indet_m]$
such that
$P(u_1,\ldots,u_m)=0.$
Otherwise, they are \demph{algebraically dependent over $E$}.

A subset $B\subseteq F$ is a \demph{transcendence basis} for $F$ over
$E$ if $B$ is algebraically independent over $E$ and the extension
$E(B) \subseteq F$ is algebraic. Any two transcendence bases have the same
cardinality. This cardinality is the \demph{transcendence degree} of
$F$ over $E$ and is denoted by
\[
\trdeg_E F.
\]
\end{definition}

\begin{example}[Nonconstant analytic functions]
\label{ex:nonconstant-analytic-functions}
If $v\in\OI$ is nonconstant, then $v$ is transcendental over $\RR$.
Indeed, if $v$ satisfies a nonzero polynomial $g\in\RR[X]$, then
\[
v(\interval)
\subseteq
\{a\in\RR:g(a)=0\}.
\]
The set on the right is finite, whereas the continuous image
$v(\interval)$ is connected. Hence $v(\interval)$ must be a
singleton, which shows that $v$ is constant. Since $\RR(v)$ is
generated by one transcendental element,
\[
\trdeg_{\RR}\RR(v)=1.\vspace{-12pt}
\] 
\end{example}

\subsection{The power-series Ax theorem}

We use the following power-series form of Ax's theorem
\cite[Corollary 1]{Ax-1971-On-Schanuel-Conjectures}. It is stated over $\CC$, as in
the standard power-series formulation. The real-analytic consequence
below follows by complexifying the Taylor-series field.

\begin{theorem}[Power-series Ax theorem]
\label{thm:ax-power-series}
Let
$u_1,\ldots,u_s
\in
\CC[[t_1,\ldots,t_r]]$
be formal power series with zero constant term. Suppose that
$u_1,\ldots,u_s$ are linearly independent over $\QQ$. Then
\begin{equation}
\label{eq:ax-power-series}
\trdeg_{\CC}
\CC
\bigl(
u_1,\ldots,u_s,e^{u_1},\ldots,e^{u_s}\bigr)
\geq
s+
\rank
\left(
\frac{\partial u_i}{\partial t_j}
\right)_{
\substack{1\leq i\leq s\\1\leq j\leq r}
},
\end{equation}
where the rank is computed over the fraction field of
$\CC[[t_1,\ldots,t_r]]$.
\end{theorem}

\begin{example}
\label{ex:ax-z-z2-z3}
Let $z$ be a formal variable and 
consider the $s=3$ formal power series

\[
u_1=z,
\qquad
u_2=z^2,
\qquad
u_3=z^3
\]
in $\CC[[z]]$.
These series have zero constant term and are linearly independent over
$\QQ$. Their Jacobian matrix is the nonzero column
\[
\begin{pmatrix}
1\\
2z\\
3z^2
\end{pmatrix},
\]
which has rank one. 
Applying
Theorem~\ref{thm:ax-power-series} gives
\[
\trdeg_{\CC}
\CC\bigl(z,z^2,z^3,e^z,e^{z^2},e^{z^3}\bigr)
\geq
3+1 = 4.
\]
    Since $z^2,z^3\in\CC(z)$ , we have
$
\CC\bigl(
z,z^2,z^3,e^z,e^{z^2},e^{z^3}
\bigr)
=
\CC\bigl(
z,e^z,e^{z^2},e^{z^3}
\bigr).
$
The field on the right is generated over $\CC$ by four elements, so
its transcendence degree is at most $4$ . Therefore,
$\trdeg_{\CC}
\CC\bigl(
z,e^z,e^{z^2},e^{z^3}
\bigr)
=
4.$ \vspace{-5pt}
\end{example}

    \subsection{A one-variable analytic consequence}

\begin{proposition}[One-variable Ax consequence] \label{prop:one-variable-ax} 
Let $v_1,\ldots,v_s\in\OI$, 
with $s\geq1.$ 
Suppose that 
$v_1',\ldots,v_s'$ 
are linearly independent over $\QQ$ . Then \begin{equation} \label{eq:ax-one-variable} \trdeg_{\CC} \CC(v_1,\ldots,v_s,e^{v_1},\ldots,e^{v_s}) \geq s+1. \end{equation} \end{proposition}

\begin{proof}
For $i=1,\ldots,s$ , set
\[
u_i
:=
v_i-v_i(t_0)\]
so that
$ u_i(t_0)=0$ . Then
\begin{equation}
\label{eq:centered-complex-exponential-fields}
\CC
\bigl(
v_1,\ldots,v_s,
 e^{v_1},\ldots,e^{v_s}
\bigr)
=
\CC
\bigl(
u_1,\ldots,u_s,
 e^{u_1},\ldots,e^{u_s}
\bigr).
\end{equation}
because
$v_i=u_i+v_i(t_0)$
$e^{v_i}=e^{v_i(t_0)}e^{u_i}.$
Moreover, the functions $u_1,\ldots,u_s$ are linearly independent over
$ \QQ$ : 
if
\[
\sum_{i=1}^s q_i u_i=0,
\qquad
q_i\in\QQ,
\]
then differentiation gives
$\sum_{i=1}^s q_i v_i'=0$
thereby implying every $q_i$ is zero by the hypothesis.

Using the Taylor embedding $T_{t_0}$, 
set the formal series
\[
\widehat u_i
:=
T_{t_0}(u_i)
\in
z\RR[[z]]
\subseteq
z\CC[[z]].
\]
Since $T_{t_0}$ is injective, 
$ \widehat u_1,\ldots,\widehat u_s$ are linearly independent over
$ \QQ$ . 

    By \eqref{eq:taylor-differentiation}, the Jacobian of these formal
power series is the $s\times1$ matrix
\[
\begin{pmatrix}
\dfrac{d\widehat u_1}{dz}\\[1ex]
\vdots\\[1ex]
\dfrac{d\widehat u_s}{dz}
\end{pmatrix}
=
\begin{pmatrix}
T_{t_0}(u_1')\\
\vdots\\
T_{t_0}(u_s')
\end{pmatrix}
=
\begin{pmatrix}
T_{t_0}(v_1')\\
\vdots\\
T_{t_0}(v_s')
\end{pmatrix}.
\]
This matrix is nonzero because the assumed $\QQ$ -linear
independence of $v_1',\ldots,v_s'$ implies that not
all of these functions vanish identically. 
Therefore, the matrix has rank one.
Applying \Cref{thm:ax-power-series},
we get
\begin{equation}
\label{eq:complexified-ax-bound}
\trdeg_{\CC}
\CC\bigl(
\widehat u_1,\ldots,\widehat u_s,
e^{\widehat u_1},\ldots,e^{\widehat u_s}
\bigr)
\geq
s+1,
\end{equation}
and the result follows
from the fact
the Taylor embedding restricts to an isomorphism from 
$\CC(u_1,\ldots,u_s,e^{u_1},\ldots,e^{u_s})$
onto 
$\CC(\widehat u_1,\ldots,\widehat u_s, e^{\widehat u_1},\ldots,e^{\widehat u_s}) \subseteq \CC((z)).$
and \eqref{eq:centered-complex-exponential-fields}.
\end{proof}

\section{The Ax--ridgeline bridge}
\label{sec:ax-ridgeline-bridge}

    The ridgeline correspondence $\RidgeCorr(f_1,\ldots,f_k)$ is algebraic, whereas the profile graph $\graphprofilemap$ is real analytic. 
    By \Cref{prop:critical-point-incidence}, their intersection encodes the critical points of the mixture. 
    Thus, an infinite critical set leads us to consider real-analytic arcs along which polynomial equations and exponential relations hold simultaneously. 
    In this section, we combine the finite-fiber geometry of the ridgeline correspondence with the transcendence degree estimate in 
    \Cref{prop:one-variable-ax}. 
    The finite-fiber condition gives an upper bound on the transcendence degree of the functions along such an arc, while Ax's theorem gives a contradictory lower bound.

\begin{lemma}[Zariski closures under finite-fiber projection]
\label{lem:zariski-closure-finite-fiber}
Let
$\closedalgebraicset
\subseteq
\CC^d\times\CC^m$
be a closed algebraic set, and let
$V
\subseteq
\CC^m$
be a Zariski-open set such that the coordinate projection
\[
\pi:
\closedalgebraicset
\cap
\bigl(\CC^d\times V\bigr)
\longrightarrow
V,
\qquad
(x,r)
\longmapsto
r,
\]
has finite fibers.
Consider the real-analytic maps
$\xPath:
\interval
\longrightarrow
\RR^d$
and
$\rho:
\interval
\longrightarrow
\RR^m$
such that
\[
\bigl(\xPath(t),\rho(t)\bigr)
\in
\closedalgebraicset
\text{ for every }t\in\interval,
\]
and
$\rho(\interval)
\subseteq
V.$
Then, the field extension
$\CC(\rho_1,\ldots,\rho_m)
\subseteq
\CC(
\xPath_1,\ldots,\xPath_d,
\rho_1,\ldots,\rho_m
)$ is finite.
\end{lemma}

\begin{proof} 
Consider the Zariski closures 
\[ 
X_\rho := \overline{\rho(\interval)}^{\,\mathrm{Zar}} \subseteq \CC^m  
\quad \text{ and }\quad  
 X_{\xPath,\rho} := \overline{ \left\{ \bigl(\xPath(t),\rho(t)\bigr): t\in\interval \right\} }^{\,\mathrm{Zar}} \subseteq \CC^d\times\CC^m. 
\]
The defining ideals of these varieties are the kernels of the evaluation homomorphisms  
\[ 
    \varphi_\rho: 
    \CC[y_1,\ldots,y_m] \longrightarrow \MIC, \qquad y_i \longmapsto \rho_i,  
\]
and  
\[ 
    \varphi_{\xPath,\rho}: 
        \CC[ x_1,\ldots,x_d, y_1,\ldots,y_m ] \longrightarrow \MIC,  
    \qquad x_j \longmapsto \xPath_j, \;
    y_i \longmapsto \rho_i.
\] 
 
Since $\MIC$ is a field, both kernels are prime. 
Hence $X_\rho$ and $X_{\xPath,\rho}$ are irreducible. 
Because  
\[ \bigl(\xPath(t),\rho(t)\bigr) \in \closedalgebraicset \qquad \text{for every }t\in\interval  
\] and $\closedalgebraicset$ is Zariski closed, we have  
    $X_{\xPath,\rho} \subseteq \closedalgebraicset.$ 
The coordinate projection restricts to a dominant morphism  
\[ 
    \pi_{\xPath,\rho}: 
        X_{\xPath,\rho} \longrightarrow X_\rho,
        \qquad (x,r) \longmapsto r. 
\] 
Indeed, the image contains 
$\rho(\interval)$, 
which is Zariski dense in $X_\rho$. 
Set  
\[ X_\rho^\circ := X_\rho\cap V  \qquad\text{ and  }\qquad
 X_{\xPath,\rho}^\circ := X_{\xPath,\rho} \cap \bigl( \CC^d\times X_\rho^\circ \bigr). 
\] 
Since $\rho(\interval)\subseteq V$, these are nonempty dense open subsets of $X_\rho$ and $X_{\xPath,\rho}$, respectively. 
For every $r\in X_\rho^\circ$, the fiber of  
\[ X_{\xPath,\rho}^\circ \longrightarrow X_\rho^\circ  
\] is contained in the fiber of  
$ \closedalgebraicset \cap \bigl( \CC^d\times V \bigr) \longrightarrow V  
$ over $r$. 
It is therefore finite. 
Hence the generic fiber of  
\[ X_{\xPath,\rho} \longrightarrow X_\rho  
\] is finite, and the induced extension of rational function fields  
\[ \CC(X_\rho) \subseteq \CC(X_{\xPath,\rho})  
\] is finite. 
Finally, 
we have
$\CC[X_\rho] \cong \CC[\rho_1,\ldots,\rho_m] $ and
$\CC[X_{\xPath,\rho}] \cong \CC[ \xPath_1,\ldots,\xPath_d, \rho_1,\ldots,\rho_m ]. $
Therefore
\[ \CC(X_\rho) \cong \CC(\rho_1,\ldots,\rho_m) \quad \text{ and } \quad  \CC(X_{\xPath,\rho}) \cong \CC( \xPath_1,\ldots,\xPath_d, \rho_1,\ldots,\rho_m ). 
\] 
Under these identifications, the induced finite extension is  
\[ \CC(\rho_1,\ldots,\rho_m) \subseteq \CC( \xPath_1,\ldots,\xPath_d, \rho_1,\ldots,\rho_m ). 
 \qedhere \] 
\end{proof}

\begin{theorem}[Ridgeline Criterion for finitely many critical points]
\label{thm:criterion}
Let $k\geq1$, and for $i=1,\dots,k$ let
\[
f_i=e^{h_i},
\qquad
h_i\in\RR[x_1,\ldots,x_d],
\]
be the components of the mixture density
$\mixturedensity
=
\sum_{i=1}^{k}w_i f_i$,
with 
$w\in \OpenSimplex$.
Assume that:
\begin{enumerate}
    \item the critical set $\Crit(\mixturedensity)$ is compact; and

    \item there exists a dense Zariski-open set
    $
    \profileopenset\subseteq
    \PP_{\CC}^{k-1}$ 
    such that
    \begin{enumerate}
        \item     $ \positiveprofiles \subseteq \profileopenset $, and
        \item     the projection 
    $\betaprojection
    {
    \RidgeCorr(f_1,\ldots,f_k)
    \cap
    \bigl(\CC^d\times \profileopenset\bigr)
    }
    {\profileopenset}
    $
    has finite fibers.

    \end{enumerate}

\end{enumerate}
Then $\Crit(\mixturedensity)$ is finite.
\end{theorem}

\begin{proof}

Fix the dense Zariski-open set $\profileopenset\subseteq\PP_{\CC}^{k-1}$ so that it satisfies~{item 2.}  

\medskip

If $k=1$, then $\PP_{\CC}^{k-1}=\PP_{\CC}^{0}$ is a point. Since
$\profileopenset$ contains its unique point, we have
$\profileopenset=\PP_{\CC}^{0}$. Thus the unique fiber of the profile
projection is $\RidgeCorr(f_1)$, which is finite by hypothesis.
Proposition~\ref{prop:critical-point-incidence} therefore implies that
$\Crit(\mixturedensity)$ is finite. We may henceforth assume that
$k\geq2$.

Suppose, toward a contradiction, that
$ \Crit(\mixturedensity)$ is infinite. 
Since it is compact, it has an accumulation point $x_0\in\Crit(\mixturedensity).$
The mixture  $\mixturedensity$ is real analytic. Hence \[ \Crit(\mixturedensity) = \left\{ x\in\RR^d: \frac{\partial\mixturedensity}{\partial x_1}(x) = \cdots = \frac{\partial\mixturedensity}{\partial x_d}(x) = 0 \right\} \] is a real-analytic set and therefore locally semianalytic. Since $x_0$ is an accumulation point, \[ x_0 \in \overline{ \Crit(\mixturedensity)\setminus\{x_0\} }. \] 
By the Curve Selection Lemma~
(\cite[Lemma 2.2.3]{valette2025subanalyticgeometry}), 
there exist
$ \varepsilon>0$ and a nonconstant real-analytic map
\[
\xPath:
(-\varepsilon,\varepsilon)
\longrightarrow
\RR^d
\]
such that
$\xPath(0)=x_0$
and
$\xPath((0,\varepsilon))
\subseteq
\Crit(\mixturedensity)\setminus\{x_0\}.$
Henceforth,
set
$\interval:=(0,\varepsilon)$
and  write $\xPath$ for the restriction of this map to
$ \interval$ . So
each coordinate function
$\xPath_j:
\interval
\longrightarrow
\RR$
belongs to $\OI$ .

We now use the critical point incidence proposition to lift this
real-analytic path of critical points to a real-analytic path in the
algebraic ridgeline correspondence.
Define the path
$
\profilePath
:=
\profileMap\circ\xPath:
\interval
\longrightarrow
\positiveprofiles
$, 
{where $\profileMap$ is as in \eqref{eq:define-profile-map-w}.}
Note that  by construction
\[
\profilePath(\interval)
\subseteq
\positiveprofiles
\subseteq
\profileopenset.
\]
Furthermore, the path
\[
\gamma
:
\interval
\longrightarrow
\RR^d\times\positiveprofiles,\qquad t\mapsto (\xPath(t),\profilePath(t))
\]
by definition, satisfies $\gamma(\interval)\subseteq\graphprofilemap$. Since
$\xPath(\interval)\subseteq\Crit(\mixturedensity)$,
Proposition~\ref{prop:critical-point-incidence} gives
\[
\gamma(\interval)
\subseteq
\RidgeCorr(f_1,\ldots,f_k)
\cap
\graphprofilemap.
\]
Since $\beta_k>0$ on $\positiveprofiles$ , work in the affine chart
$ \beta_k\neq0$ in 
$\PP_{\CC}^{k-1}$.
Set
\[
\rho_i:=\frac{\beta_i}{\beta_k}
=
\frac{w_i}{w_k}
\exp\!\left((h_i-h_k)\circ\xPath\right),
\qquad i=1,\ldots,k-1,
\]
and write $\rho := (\rho_1,\ldots,\rho_{k-1}) : \interval \longrightarrow \RR_{>0}^{k-1}.$
Let $\RidgeCorr_{\mathrm{aff}}\subseteq\CC^d\times\CC^{k-1}$ be the ridgeline correspondence in this chart, and let
$ \profileopenset_{\mathrm{aff}}\subseteq\CC^{k-1}$ be the affine image of
$ \profileopenset\cap\{\beta_k\neq0\}$ . Then
\[
(\xPath(t),\rho(t))\in\RidgeCorr_{\mathrm{aff}},
\qquad
\rho(\interval)\subseteq \profileopenset_{\mathrm{aff}}.
\]
The projection
\[
\RidgeCorr_{\mathrm{aff}}
\cap(\CC^d\times \profileopenset_{\mathrm{aff}})
\longrightarrow \profileopenset_{\mathrm{aff}}
\]
has finite fibers by hypothesis~\textup{(2b)}. Therefore,
Lemma~\ref{lem:zariski-closure-finite-fiber} gives a finite extension
\[
K:=\CC(\rho_1,\ldots,\rho_{k-1})
\subseteq
\CC(\xPath_1,\ldots,\xPath_d,
\rho_1,\ldots,\rho_{k-1}).
\]
    The extension being finite implies that every coordinate $\xPath_j$ is algebraic over $K$. 
    Moreover, since $(h_i-h_k)\circ\xPath \in \CC[\xPath_1,\ldots,\xPath_d]$,  every  function $(h_i-h_k)\circ\xPath$ is also algebraic over $K$.

\medskip

For $i=1,\ldots,k-1$ , set
$H_i:=(h_i-h_k)\circ\xPath\in\OI.$
Consider the finite-dimensional $\QQ$ -vector space
spanned by the derivatives of $H_i$:
\begin{equation}
\label{eq:derivative-span}
W_h
:=
\operatorname{span}_{\QQ}
\{H_1',\ldots,H_{k-1}'\}
\subseteq\OI.
\end{equation}
We denote the  
$\mathbb{Q}$-vector space dimension by $s$.

If $s=0$ , then every $H_i$ , and hence every $\rho_i$ , is constant. The path $(\xPath,\rho)$ therefore lies in a single finite profile fiber. Since $\xPath$ is continuous, it is constant, a contradiction.

Therefore,
we may choose $v_1,\dots,v_s$, 
to be a subset of $H_1,\dots, H_{k-1}$ 
such that the derivatives
$v_1',\dots,v_s'$ form a basis for $W_h$.

Because $v_1',\ldots,v_s'$ form a basis of $W_h$ ,
there exist unique rational numbers
$ q_{i1},\ldots,q_{is}\in\QQ$ such that
\begin{equation}
\label{eq:derivative-basis-expansion}
H_i'
=
\sum_{j=1}^{s}q_{ij}v_j',\qquad i=1,\dots,k-1.
\end{equation}
Therefore, 
for $i=1,\dots,k-1$,
$\left(
H_i-
\sum_{j=1}^{s}q_{ij}v_j
\right)'=0$
and
there exists a constant $c_i\in\RR$ such
that
$H_i -
\sum_{j=1}^{s}q_{ij}v_j
=
c_i
.$
Equivalently, 
\begin{equation}\label{eq:relation-h-left-v-right}
(h_i-h_k)\circ\xPath
=
c_i+
\sum_{j=1}^{s}q_{ij}v_j.
\end{equation}

We set
$K_0
:=
\CC(e^{v_1},\ldots,e^{v_s})$
and have the field extension
\[L
:=
\CC(
v_1,\ldots,v_s,
e^{v_1},\ldots,e^{v_s}
).
\]
Choose $N_i>0$ such that $N_iq_{ij}\in\ZZ$ for every $j$. Then, from \eqref{eq:relation-h-left-v-right}, we derive the relation
\[
\left(e^{(h_i-h_k)\circ\xPath}\right)^{N_i}
=
e^{N_ic_i}
\prod_{j=1}^{s}(e^{v_j})^{N_iq_{ij}}
\in
K_0.
\]
Thus every $e^{(h_i-h_k)\circ\xPath}$ is algebraic over $K_0$ (as an  $N_i$-th root). 
So
$K$ is algebraic over $K_0$. Each $v_j$ is algebraic over $K$ and hence over
$K_0$. Therefore $L$ is algebraic over $K_0$, and
\[
\trdeg_{\CC}L
=
\trdeg_{\CC}K_0
\leq
s.
\]
On the other hand,
$ v_1',\ldots,v_s'$ are linearly independent over $\QQ$ by
construction. Proposition~\ref{prop:one-variable-ax} therefore gives
\[
\trdeg_{\CC}L
=
\trdeg_{\CC}
\CC(v_1,\ldots,v_s,e^{v_1},\ldots,e^{v_s})
\geq
s+1.
\]
This contradicts the upper bound
\[
\trdeg_{\CC}L\leq s.
\]
Therefore, $\Crit(\mixturedensity)$ is finite.
\end{proof}
One way to verify compactness of the
critical set is by extending a positive ridgeline locus parameterization from
the open simplex $\Delta_{k-1}^{\circ}$ to the closed simplex
$ \Delta_{k-1}$ , thus parameterizing the \demph{nonnegative ridgeline locus}  $\mathcal R_{\geq0}(f_1,\ldots,f_k)$.

\begin{proposition}[Compactness from the nonnegative ridgeline locus]
\label{prop:compactness-via-closed-ridgeline}
Suppose the function $R:\OpenSimplex\longrightarrow\RR^d$ parameterizes the positive ridgeline locus as $$\mathcal R_{>0}(f_1,\ldots,f_k)=R(\Delta_{k-1}^{\circ}),$$ 
and  extends to a continuous map
$R:\Delta_{k-1}\longrightarrow\RR^d.$
Then $\Crit(\mixturedensity)$ is compact.
\end{proposition}
\begin{proof}
Since
$ \mixturedensity$ 
is real analytic,
$ \nabla\mixturedensity$ is continuous, and
\[
\Crit(\mixturedensity)
=
(\nabla\mixturedensity)^{-1}(\{0\}),
\]
we conclude that $\Crit(\mixturedensity)$ is a closed set. By
\Cref{cor:ridgeline-containment},
\[
\Crit(\mixturedensity)
\subseteq
\mathcal R_{>0}(f_1,\ldots,f_k).
\]
We also have that
\[ \mathcal R_{>0}(f_1,\ldots,f_k)
= R(\Delta_{k-1}^{\circ}) \subseteq R(\Delta_{k-1}) = \mathcal R_{\geq0}(f_1,\ldots,f_k).\]
Since $\Delta_{k-1}$ is compact and $R$ is continuous, $R(\Delta_{k-1})=\mathcal R_{\geq0}(f_1,\ldots,f_k)$ is compact. We have shown that $\Crit(\mixturedensity)$ is a closed subset of $R(\Delta_{k-1})$ . Therefore, $\Crit(\mixturedensity)$ is~compact.
\end{proof}

    \section{Application to Gaussian mixtures}
\label{sec:gaussian-mixtures}

Consider the (homoscedastic or heteroscedastic) Gaussian mixture density
\[
\mixturedensity(x)
=
\sum_{i=1}^{k} w_i f_i(x),
\qquad
x\in\RR^d,
\]
where $w_i>0$, $\sum_{i=1}^{k}w_i=1$, and
\[
f_i(x)
=
\frac{1}{\sqrt{\det(2\pi\Sigma_i)}}
\exp\left(
-\frac12
(x-\mu_i)^{\mathsf T}\Sigma_i^{-1}(x-\mu_i)
\right).
\]
Here $\mu_i\in\RR^d$, and each $\Sigma_i$ is real symmetric and
positive definite. Writing 
$f_i=e^{h_i}$, we have
\[
h_i(x)
=
-\frac12
(x-\mu_i)^{\mathsf T}\Sigma_i^{-1}(x-\mu_i)
-\frac12\log\det(2\pi\Sigma_i)
\]
and
\[
\nabla h_i(x)=-\Sigma_i^{-1}(x-\mu_i).
\]
For $\beta=(\beta_1,\ldots,\beta_k)\in\CC^k$, define
\[
A(\beta):=\sum_{i=1}^{k}\beta_i\Sigma_i^{-1},
\qquad
b(\beta):=\sum_{i=1}^{k}\beta_i\Sigma_i^{-1}\mu_i.
\]
The unsaturated Gaussian ridgeline equations are
$\sum_{i=1}^{k}\beta_i\nabla h_i(x)=0$,
or equivalently
\begin{equation}
\label{eq:gaussian-ridgeline-equations}
A(\beta)x=b(\beta).
\end{equation}
Take  
$\profileopenset\subseteq \PP^{k-1}_\CC$ to be the dense Zariski open set 
\begin{equation}\label{eq:define-U-heteroskedastic}
\profileopenset:=\{
[\beta]\in \PP^{k-1}_\CC:
\det A(\beta)\neq 0
\}
\end{equation}

\begin{proposition}[Gaussian profile fibers and compactness]
\label{prop:compact-and-gaussian-profile-fibers}
The open set $\profileopenset$ in
\eqref{eq:define-U-heteroskedastic} is dense and contains
$ \positiveprofiles$ . Moreover, the profile projection
\[
\betaprojection
{
\RidgeCorr(f_1,\ldots,f_k)
\cap
(\CC^d\times\profileopenset)
}
{\profileopenset}
\]
has fibers of cardinality one, and
$ \Crit(\mixturedensity)$ is compact.
\end{proposition}
\begin{proof}
Let $[\beta]\in\positiveprofiles$ . Since each $\Sigma_i^{-1}$ is positive definite,
then $A(\beta)$ is positive definite, and
$\positiveprofiles
\subseteq
\profileopenset.$
Since $A(\beta)$ is invertible,
the fiber consists of the unique point
\begin{equation}
\label{eq:gaussian-ridgeline-parameterization}
x
=
A(\beta)^{-1}b(\beta).
\end{equation}
The positive ridgeline locus is then parameterized by $\OpenSimplex$ through \eqref{eq:gaussian-ridgeline-parameterization}:
$$\mathcal R_{>0}(f_1,\ldots,f_k)
 = \{ A(\weight)^{-1}b(\weight) : \weight\in\OpenSimplex  \}$$
This parametrization can be extended to $\Delta_{k-1}$ since $A(w)$ remains invertible as long as $w\in\Delta_{k-1}$. Hence compactness of $\Crit(\mixturedensity)$ follows from \Cref{prop:compactness-via-closed-ridgeline}
\end{proof}

We can now prove \Cref{theorem:quadratic}.

\begin{proof}[Proof of Theorem~\ref{theorem:quadratic}]
By
~\Cref{prop:compact-and-gaussian-profile-fibers}, 
the critical set
$\Crit(\mixturedensity)$ is compact;
and there is a dense
Zariski-open set $\profileopenset\subseteq\PP_{\CC}^{k-1}$ containing
$\positiveprofiles$ such that the profile projection over $\profileopenset$ has finite
fibers.
\Cref{thm:criterion} implies that
$\Crit(\mixturedensity)$ is finite. 
Therefore 
 every Gaussian mixture has
finitely many modes. 
 \end{proof}

\begin{remark}[The homoscedastic case]
If $\Sigma_1=\cdots=\Sigma_k$, then we have
\[ 
    A(\weight)^{-1}b(\weight) 
    = 
    \sum_{i=1}^{k}\weight_i\mu_i. 
\] 
Hence the closed ridgeline locus becomes 
$    \mathcal R_{\geq0}(f_1,\ldots,f_k) 
    =
    \operatorname{conv}\{\mu_1,\ldots,\mu_k\}.$  
In particular, one recovers the well-known fact \cite{carreira2003number,Ray-Lindsay} that for homoscedastic Gaussian mixtures,
\[ 
    \Crit(\mixturedensity) \subseteq \operatorname{conv}\{\mu_1,\ldots,\mu_k\}. 
\] 
\end{remark}

\section{Discussion and conclusion}

\subsection{Global maximizers and finite support of Gaussian NPMLEs}
\label{ss:global-maximizers-and-npmles}

As an immediate consequence of \Cref{theorem:quadratic}, every finite Gaussian mixture has only finitely many global maximizers. 
    \begin{corollary}[Finiteness of global maximizers] \label{cor:finite-global-maximizers} Let 
\[ 
    f(x) = 
    \sum_{i=1}^{k} 
    w_i 
    \frac{1}{\sqrt{\det(2\pi\Sigma_i)}} 
    \exp\left( 
        -\frac12 (x-\mu_i)^{\mathsf T}\Sigma_i^{-1}(x-\mu_i)
    \right), 
    \qquad x\in\RR^d, 
    \] where $w_i>0$
    and $\Sigma_i$ is positive definite. 
    Then 
    $\operatorname*{arg\,max}_{x\in\RR^d} f(x)$
    is a nonempty finite set. 
    \end{corollary}

We next apply 
\Cref{cor:finite-global-maximizers}
to  \demph{nonparametric maximum likelihood estimation} for \demph{Gaussian location mixtures}.
Following the notation in \cite{NPMLE},
let $X_1,\ldots,X_n\in\RR^d$, and let
$\Sigma_1,\ldots,\Sigma_n$ be positive definite covariance matrices. Let $\mathcal{P}(\RR^d)$ denote the space of all probability measures on $\RR^d$. 
For $G \in \mathcal{P}(\RR^d)$, write
\[
f_{G,\Sigma_i}(x)
:=
\int_{\RR^d}
\phi_{\Sigma_i}(x-\theta)\,dG(\theta),
\]
where $\phi_{\Sigma_i}$ denotes the density of $N(0,\Sigma_i)$. A \demph{Gaussian NPMLE} is any solution of
\begin{equation}
\label{eq:heteroscedastic-gaussian-npmle}
\widehat G
\in
\operatorname*{arg\,max}_{G\in\mathcal P(\RR^d)}
\frac{1}{n}
\sum_{i=1}^{n}
\log f_{G,\Sigma_i}(X_i).
\end{equation}

\begin{corollary}[Finite support of Gaussian NPMLEs]
\label{cor:gaussian-npmle-finite-support}
Every solution $\widehat G$ of
\eqref{eq:heteroscedastic-gaussian-npmle} has finite support.
Moreover, there is a finite set $\mathcal A\subseteq\RR^d$ such that
$\operatorname{supp}(\widehat G)\subseteq\mathcal A$
for any solution $\widehat G$.
\end{corollary}

\begin{proof}
For a solution $\widehat G$, set
\[
\widehat L_i
:=
f_{\widehat G,\Sigma_i}(X_i),
\qquad i=1,\ldots,n,
\]
and define the \emph{dual} Gaussian mixture
\[
\widehat\psi_n(\theta)
:=
\sum_{i=1}^{n}
\frac{\widehat L_i^{-1}}
     {\sum_{j=1}^{n}\widehat L_j^{-1}}
\phi_{\Sigma_i}(X_i-\theta), \quad \theta \in \RR^d.
\]
By \cite[Lemma~1]{NPMLE}, the fitted likelihood vector
$(\widehat L_1,\ldots,\widehat L_n)$ is independent of the choice of
NPMLE, and every solution satisfies
\[
\operatorname{supp}(\widehat G)
\subseteq
\operatorname*{arg\,max}_{\theta\in\RR^d}
\widehat\psi_n(\theta).
\]
The function $\widehat\psi_n$ is a finite (possibly heteroscedastic) Gaussian
mixture with strictly positive weights. Therefore,
\Cref{cor:finite-global-maximizers} implies that the common set
\[
\mathcal A
:=
\operatorname*{arg\,max}_{\theta\in\RR^d}
\widehat\psi_n(\theta)
\]
is finite. It follows that every NPMLE is finitely supported and that
all NPMLEs are supported on the same finite set $\mathcal A$.
\end{proof}

\begin{remark}[The ridgeline manifold and the NPMLE]
The connection with ridgeline theory is also mentioned in
\cite[Lemma 3]{NPMLE}, where
the authors observe that the support of every NPMLE is
contained in the ridgeline manifold associated with the Gaussian
components $f_i$ with mean $\mu_i = X_i$ and covariance matrix $\Sigma_i$. 
This is a special case of
\Cref{cor:ridgeline-containment}.
More recently, \cite[Theorem 2]{wang2026finitegaussianmixturesfiniteness}
was the first to show the finiteness of the support of NPMLEs in the case of homoscedastic mixtures.
\end{remark}

\subsection{Applying the criterion more broadly}
\label{ss:beyond-gaussian-mixtures}

\Cref{thm:criterion}
is not specific to Gaussian mixtures.
As a nonquadratic application,
we consider mixtures of separable
even-power exponential densities.

\begin{theorem}
\label{thm:separable-even-power-mixtures}
For $i=1,\ldots,k$, let
\begin{equation}
\label{eq:f_i-separable-even-power}
  f_i(x)
  =
  c_i
  \exp\left(
    -\sum_{j=1}^{d}
    a_{ij}(x_j-\mu_{ij})^{2m_{ij}}
  \right),
  \qquad x\in\RR^d,
\end{equation}
where $a_{ij}>0$, $m_{ij}\in\ZZ_{>0}$, $\mu_{ij}\in\RR$, and
$c_i>0$ is the normalizing constant. For $w\in\OpenSimplex$, set
\[
  \mixturedensity(x)
  =
  \sum_{i=1}^{k}w_i f_i(x).
\]
Then $\Crit(\mixturedensity)$ is finite. In particular,
$\mixturedensity$ has finitely many modes.
\end{theorem}

\begin{proof}
Write $f_i=e^{h_i}$, where
$  h_i(x)
  =
  \log c_i
  -
  \sum_{j=1}^{d}
  a_{ij}(x_j-\mu_{ij})^{2m_{ij}}.$
Then
\[
  \frac{\partial h_i}{\partial x_j}(x)
  =
  -2m_{ij}a_{ij}
  (x_j-\mu_{ij})^{2m_{ij}-1}.
\]
We first verify that $\Crit(\mixturedensity)$ is compact.
Let $x\in\Crit(\mixturedensity)$ and set
\[
  \profilePath_i:=w_i f_i(x)>0,
  \qquad i=1,\ldots,k.
\]
For $ j=1,\ldots,d$,
after dividing the $j$-th critical point equation 
we obtain
the ridgeline correspondence polynomials,
\begin{equation}
\label{eq:even-power-critical}
 G_j(x,\profilePath)= \sum_{i=1}^{k}
  \profilePath_i m_{ij}a_{ij}
  (x_j-\mu_{ij})^{2m_{ij}-1}.
\end{equation}
Note that the left hand side is a 
sum of univariate odd degree polynomial in $x_j$ with $\beta_im_{ij}a_{ij}$ positive. If $x_j^*$ is a root of this polynomial, 
then  it must be in the interval
\[
\left[
\min_{1\leq i\leq k}\mu_{ij}, \;
\max_{1\leq i\leq k}\mu_{ij}
\right].
\]
It follows the set of critical points is bounded:
\[ \Crit(\mixturedensity) \subseteq 
\left[
\min_{1\leq i\leq k}\mu_{i1}, \;
\max_{1\leq i\leq k}\mu_{i1}
\right]
\times \cdots 
\times
\left[
\min_{1\leq i\leq k}\mu_{id}, \;
\max_{1\leq i\leq k}\mu_{id}
\right]\subset \RR^d
\]
Finally, since
  $\Crit(\mixturedensity)
  =
  (\nabla\mixturedensity)^{-1}(\{0\}),$
 is closed, compactness follows.

\medskip

We next verify the other condition in \Cref{thm:criterion}. The ridgeline correspondence is defined by the equations
\begin{equation}
\label{eq:even-power-ridgeline}
  \sum_{i=1}^{k}
  \varw_i m_{ij}a_{ij}
  (\varx_j-\mu_{ij})^{2m_{ij}-1}
  =
  0,
  \qquad j=1,\ldots,d.
\end{equation}
Notice that the equation \eqref{eq:even-power-ridgeline} is a univariate polynomial in the spatial coordinate
$\varx_j$,  with coefficients in $\CC[\varw]$ as $a_{i,j}\in\RR$ and $m_{ij}\in \ZZ$.
The leading coefficients of this univariate polynomial is 
the homogeneous linear form
\[
  \lambda_j(\varw)
  :=
  \sum_{\substack{1\leq i\leq k\\m_{ij}=M_j}}
  \varw_i M_j a_{ij}
  \text{ where }
M_j:=\max_{1\leq i\leq k}m_{ij}.  
\]
 If this coefficient is nonzero, then the system of equations has finitely many solutions \cite[Finiteness Theorem]{CLO-2025}.
In other words, we have the dense Zariski-open subset of $\PP^{k-1}_\CC$,
\[
  \profileopenset
  :=
\left\{
    [\varw]\in\PP^{k-1}_\CC\,:\,
    \lambda_1(\varw)\cdot\ldots\cdot \lambda_d(\varw)\neq 0
  \right\},
\]
where the projection has finite fibers. 
Moreover, 
every positive profile belongs to this open set as for every $j$ and positive $\beta$
\[
  \lambda_j(\beta)
  =
  \sum_{\substack{1\leq i\leq k\\m_{ij}=M_j}}
  \beta_i M_j a_{ij}
  >0.
\]

We have verified both hypotheses of \Cref{thm:criterion}. 
It follows that $\Crit(f)$ is finite. Since every mode of the
smooth density $f$ is a critical point, $f$ has finitely many modes.
\end{proof}

\begin{remark} 
In the case 
    when $m_{ij}=1$ for every $i$ and $j$, 
    the component $f_i$ is a multivariate Gaussian density with mean 
    $\mu_i=(\mu_{i1},\ldots,\mu_{id})$
    and diagonal covariance matrix 
    $\Sigma_i = \operatorname{diag}\left( \frac{1}{2a_{i1}},\ldots,\frac{1}{2a_{id}} \right)$. 
 Thus, the theorem contains mixtures of Gaussian densities with diagonal covariance matrices as a special case. 
    \end{remark}

\subsection{An affine Gaussian combination with infinitely many modes}
\label{ss:affine-gaussian-circle-modes}
For $a>0$, let
\[
\varphi_a(x)
=
\frac{a}{\pi}
\exp\left(-a\lVert x\rVert^2\right),
\qquad x\in\RR^2
\]
be the density of the centered Gaussian distribution
$N\left(0,\frac{1}{2a}I_2\right)$.

Fix $0<a<b$ and consider the affine linear combination
\begin{equation}
\label{eq:affine-gaussian-density}
f(x)
=
\frac{b}{b-a}\varphi_a(x)
-
\frac{a}{b-a}\varphi_b(x).
\end{equation}
Then \eqref{eq:affine-gaussian-density} is an affine combination of
Gaussian densities,
although it is not a convex combination because
its second coefficient is negative.

Every point on the circle
\begin{equation}
\label{eq:affine-gaussian-mode-circle}
\left\{
x\in\RR^2:
\lVert x\rVert^2
=
\frac{\log(b/a)}{b-a}
\right\}
\end{equation}
is a mode of $f$, see Figure~\ref{fig:degree-4}.
However, 
this example does not contradict  \Cref{thm:criterion} because of the requirement of
strictly positive weights.

\begin{figure}[htb!]
    \centering
    \includegraphics[width=0.5\linewidth]{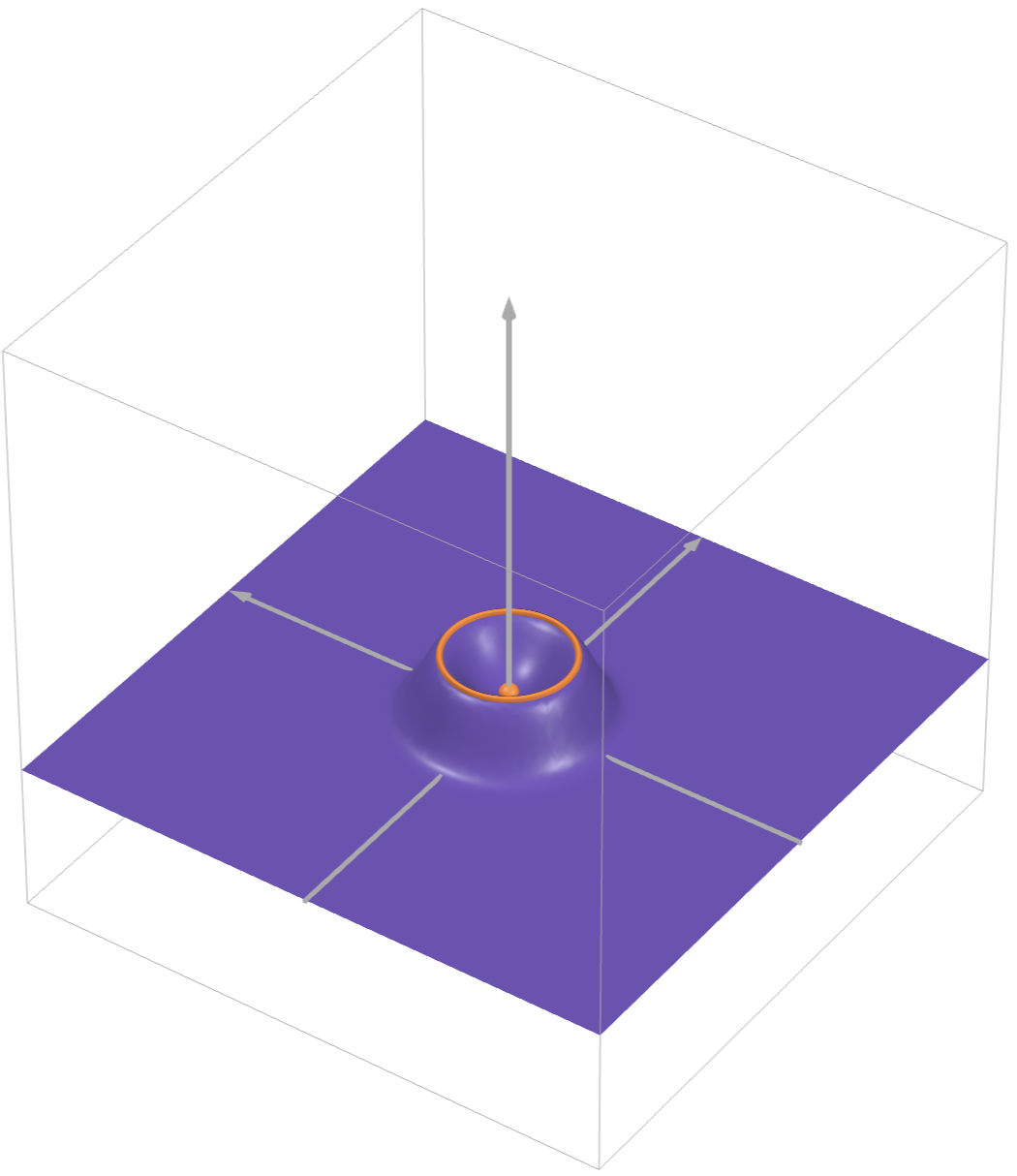}
    \caption{A density with infinitely many modes}
    \label{fig:degree-4} \end{figure}

\subsection{Radial polynomial density with infinitely many modes} 
In this example, the compactness hypothesis of \Cref{thm:criterion} holds, but the finite-fiber hypothesis 2b fails.
Consider the density
\[ 
f(x)
    =
    c_d\exp\left(-\bigl(\lVert x\rVert^2-1\bigr)^2\right) 
\] on $\RR^d$, where $d\geq2$ and $c_d>0$ is a normalizing constant. The radial polynomial $(\Vert x \Vert^2-1)^2$ is coercive, i.e., it grows to $\infty$ as $||x||\rightarrow\infty$, and yet the critical set $\Crit(f)$ is infinite:
\[ 
    \Crit(f)=\{0\}\cup
    \{
    x\in \RR^d
    : \lVert x\rVert^2=1
    \}
    \] 
is compact, but every point of the unit sphere is a mode, 
so the set of critical points and the set of modes are both infinite. The obstruction is visible algebraically. In the one-component case, the ridgeline equation over a nonzero profile $\beta$ reduces to \[ \bigl(x_1^2+\cdots+x_d^2-1\bigr)x_j=0, \qquad j=1,\ldots,d. \] Its complex solution set contains the quadric \[ x_1^2+\cdots+x_d^2=1, \] and hence has positive dimension. The profile projection 
does not have finite fibers. Therefore, 
\Cref{thm:criterion} does not apply because of item 2b.

\subsection{The maximum number of modes as an open problem}
\label{ss:extremal-number-of-modes}

For positive integers $d$ and $k$, 
recall from the introduction that
\begin{equation}
\label{eq:define-extremal-mode-number}
m(d,k)
:=
\sup
\left\{
\#\operatorname{Modes}(f):
 f \text{ is a Gaussian mixture on $\RR^d$ with $k$ components}
\right\}.
\end{equation}
\Cref{theorem:quadratic} shows 
that every Gaussian mixture has finitely many modes. 
In principle, one could vary the weights, means, and covariance matrices while keeping $d$ and $k$ fixed and obtain mixtures with arbitrarily finitely many modes and the supremum \eqref{eq:define-extremal-mode-number} would be infinity. 
However, the existence of a 
uniform bound on the number 
of isolated critical points, first proved in \cite{Amendola-Engstrom-Haase}, combined 
with 
\Cref{theorem:quadratic} 
yields that $m(d,k)<\infty$.
\begin{corollary}[Finiteness of the maximum number of modes] \label{cor:finiteness-mdk} For every $d,k\in\ZZ_{>0}$, the maximum number of modes of a $d$-dimensional Gaussian mixture with $k$ components is finite. 
In other words, 
$m(d,k)$ is finite. 

\end{corollary}

While the finiteness of $m(d,k)$ is now established, determining its value remains open.

\begin{problem}[The maximal Gaussian mode problem]
\label{prob:extremal-gaussian-modes}
Determine $m(d,k)$, or obtain upper and lower bounds that are sharp in
their dependence on $d$ and $k$.
\end{problem}
As discussed in the introduction, exact values are known in only a few cases, namely
\[
m(1,k)=k,
\qquad
m(d,1)=1,
\qquad
m(d,2)=d+1,
\]
while the first unknown value is currently  \cite{kabata2026least,nguyen2026bounds} bounded as 
$7 \leq m(2,3) \leq 196.$ 
More generally, it is interesting to determine the growth of $m(d,k)$ as either $d$ or $k$ increases and to identify the geometric configurations that produce  many modes.
Analogous questions can be posed for homoscedastic Gaussian mixtures and for the other polynomial--exponential mixture families covered by our finiteness criterion.

\subsubsection*{Acknowledgments}
Rodriguez's research is partially supported by the Alfred P. Sloan Foundation and National Science Foundation Grant No. 2510307. Rodriguez also gratefully acknowledges support from a Mercator Fellowship through the DFG Priority Programme Combinatorial Synergies (SPP 2458) and a KTH Digital Futures Scholar-in-Residence appointment.

\bigskip \medskip \bigskip

\noindent
\footnotesize {\bf Authors' addresses:}
\smallskip

\noindent Carlos Améndola, Technische Universität Berlin, Germany \hfill {\tt  amendola@math.tu-berlin.de} \url{https://www.tu.berlin/alg-geom-data}

\smallskip
\noindent Jose Israel Rodriguez, University of Wisconsin--Madison, USA \hfill {\tt  jose@math.wisc.edu}\newline
\url{https://sites.google.com/wisc.edu/jose/}

\end{document}